\documentclass[12pt]{article}

\usepackage{amssymb}
\usepackage{amsmath}
\usepackage{amsthm}

\def\ZZ{{\mathbf Z}}
\def\NN{{\mathbf N}}

\def\hpsi{{\hat \psi}}

\def\md{{\rm mod}\,}

\newtheorem{theorem}{Theorem}
\newtheorem{lemma}{Lemma}
\newtheorem{corollary}{Corollary}
\newtheorem{proposition}{Proposition}

\def\beqns{\begin{eqnarray*}}
\def\eeqns{\end{eqnarray*}}
\def\beqn{\begin{eqnarray}}
\def\eeqn{\end{eqnarray}}

\def\C{{\bf C} \,}

\newcommand{\EQ}[2]{\begin{equation}{{#2}\label{#1}} \end{equation}}

\newenvironment{thm}[2]{\begin{sloppypar}\refstepcounter{theorem}%
                        {\bf #1 \thetheorem.}\label{#2}\em{}}%
                        {\end{sloppypar}}

                        \newcommand{\R}{{\rm I}\!{\rm R}}

\begin{document}

\title{Convergence of Multilevel Stationary Gaussian Quasi-Interpolation}
\author{Jeremy Levesley\footnote{Department of Mathematics, University of Leicester, LE1 7RH, UK. {\tt jl1@le.ac.uk}} and Simon Hubbert\footnote{Department of Economics, Mathematics and Statisitcs, Birkbeck, University of London, WC1H 7HX, UK. {\tt s.hubbert@bbk.ac.uk}}}

\maketitle

\begin{abstract}
In this paper we present a new multilevel quasi-interpolation algorithm for smooth periodic functions using scaled Gaussians as basis functions. Recent research in this area has focussed upon implementations using basis function with finite smoothness. In this paper
we deliver a first error estimates for the multilevel algorithm using analytic basis functions. The estimate  has two parts, one involving the convergence of a low degree polynomial truncation term  and  one involving  the control of the remainder of the truncation as the algorithm proceeds. Thus, numerically one observes a convergent scheme. Numerical results suggest that the scheme converges much faster than the theory shows.
\end{abstract}

\section{Introduction}

The radial basis function (RBF) method has become a successful tool for approximating functions from scattered data. However,  despite many promising theoretical advances one drawback is that for very large data sets the method struggles to maintain a good fit in a numerically stable manner.
To overcome this problem Floater and Iske \cite{floater} proposed a multilevel approximation method where an initial stable approximation is formed on a relatively sparse subset of the data and this is   then refined over multiple levels of residual RBF interpolation on progressively denser subsets.  The original implementation uses Wendland RBFs (finitely smooth and compactly supported) where the size of the support is scaled to reflect the relative density at a given level. In \cite{hales} a multilevel scheme using polyharmonic splines (finitely smooth and globally supported) on uniform grids was presented and constant reduction in error per level was shown. In \cite{iske}  a modified multilevel method was considered, using thin-plate splines for an initial approximation and with subsequent refinements performed using scaled Wendland RBFs. Wendland and coauthors have explored multilevel schemes using scaled Wendland RBFs for solving both approximation problems and partial differential equations on spheres and compact regions in Euclidean space \cite{farrell,legia1,legia2,wendland}. A hurdle in proving convergence results is that by changing scale of the basis function we also change approximation spaces however, in relation to this,  we highlight the work of Narcowich et al. \cite{narcowich} who analysed a related scheme but required that sequences of approximation spaces were nested.

 As far as the authors are aware the extant  theoretical results on the multilevel method (briefly reviewed in the previous paragraph) apply only to basis functions with finite smoothness. In these cases the numerical stability is improved but one has to accept a saturation point on the accuracy. However, recently  multi-level approximation using scaled Gaussians (infinitely smooth and globally supported) has become of interest due to its key role in multilevel sparse kernel interpolation (MuSIK) and its quasi-interpolatory modification (Q-MuSIK), see \cite{georgoulis,usta}. These approaches have achieved successful results in different areas, see \cite{dong,zhao,usta} for details, and their success provides the motivation for this current work.
 Specifically, our aim here is to present a first convergence analysis of the multilevel approximation method  using the Gaussian  basis function. The approach we take  differs from the standard formulation in that we replace interpolation with quasi-interpolation. In order to make the analysis tractable we  will investigate the performance of the scheme when approximating univariate real valued functions with period one. The classical approach to this approximation problem is to use Fourier series, but over the past 50 years, many authors have used shifts of a univariate function \cite{golomb,kushpel,pinkus} and this approach has been adapted to the torus \cite{gomes} and the sphere \cite{wahba}.
 
The paper is organized as follows. In Section 2 we provide a precise statement of the problem we want to solve together with a description of the proposed  multilevel solution method with Gaussian quasi-interpolation. In addition, we will compose the key mathematical results that will be useful in the subsequent analysis. In Section 3 we will develop convergence estimates for the full algorithm for even functions (a sum of cosines). The proof for odd functions is similar. As in the theory of approximate approximation developed by Maz'ya and Schmidt \cite{mazya}, this error will have a part that is reducing at a fixed rate with each iteration, and a part which starts off extremely small (which we will call $\epsilon$) but will grow with a fixed rate with each iteration. In Section~\ref{numerics} we will present numerical examples.

  \section{Background and Preliminaries} \label{preliminaries}

Following Delvos \cite{delvos}, we let $\mathcal{C}$ denote the space of continuous real-valued function with period one which we equip with the uniform norm $\|f\|_{\infty}=\sup_{x \in \R}|f(x)|.$ Next we let $\mathcal{L}_{2}$
denote the Hilbert space of square integrable periodic functions with inner product
\[
(f,g):=\int_{0}^{1}f(x)g(x)dx.
\]
The exponentials are given by $e_{k}(x)=\exp(2\pi i k x)$ for $k \in \ZZ.$ The finite Fourier transform of $f \in  \mathcal{L}_{2}$ is given by $\widehat{f}_{k}=(f,e_{-k})$ for $k \in \ZZ$ and its inversion is the Fourier series of $f$ given by $\sum_{k=-\infty}^{\infty}\widehat{f}_{k}e_{k},$
which converges to $f$ in the $\mathcal{L}_{2}$-norm $\|\cdot\|_{2}$ induced by the inner product. Next we let $\mathcal{N}$ denote the space of functions $f \in  \mathcal{L}_{2}$ having absolutely convergent Fourier series, i.e., those for which the norm $\|f\|=\sum_{k=-\infty}^{\infty}|\widehat{f}_{k}|$
is finite. We have the inclusions $\mathcal{N} \subset \mathcal{C} \subset \mathcal{L}_{2}$ and so for any $f \in \mathcal{N}$ we have the estimates $\|f\|_{2}\le \|f\|_{\infty} \le \|f\|.$

In our convergence analysis we will need to measure the smoothness of our target functions more precisely and so,
to complete the review of pertinent function spaces, we shall also consider the periodic Sobolev spaces of order $s:$
\[
\mathcal{W}_s = \left \{ f\in \mathcal{L}_{2}: \| f \|_s = \left ( |\widehat{f}_{0}|^2 + \sum_{k \in \ZZ}^\infty k^{2s} |\widehat{f}_{k}|^2 \right )^{1/2} < \infty \right \}.
\]
For $s>\frac{1}{2}$ the Sobolev space  $\mathcal{W}_s$ is a subspace of $\mathcal{N}$ and therefore also of  $\mathcal{C}$.

We can now turn to the approximation problem which, at the most general level, can be stated as follows: for a general target function $f,$ construct an approximating function $s_{f}$ based on the data set $\{f(h\ell):\ell\in \ZZ \,\,\,{\rm{and}} \,\,\, 0<h<1\}.$ In our work we will  consider target functions with period one taken from either $\mathcal{N}$ or from an appropriate Sobolev space $\mathcal{W}_s$ with $ s>\frac{1}{2}.$ We will construct our approximation via Schoenberg's approach \cite{schoen} to quasi-interpolation with a Gaussian as the underlying basis function. Specifically, we will consider the following stationary quasi-interpolant
\EQ{quasi}{
Q_{h}(f)(x):=\sum_{\ell \in \ZZ}f(h\ell)\psi\left(\frac{x}{h}-\ell\right)\,\,\,\,{\rm{where}}\,\,\,\,
\psi(x)= \frac{1}{\sqrt{2\pi}}\exp\left(-\frac{x^{2}}{2}\right).
}

The implementation of the proposed method is as follows. We fix the initial (level one) set of sample points by choosing an appropriate integer $\ell$ and setting $ h= \frac{1}{2^{\ell}}.$ We then form the quasi-interpolant (\ref{quasi}) to the target function $f.$ As we move from one level to the next the spacing between the sample points decreases by a factor of $1/2,$ thus at level $p$ say the spacing is $h/2^{p-1}.$ At each subsequent level (beyond the first one) we form the quasi-interpolant to the residual function (from the previous stage) and this is then added to the current approximation. Continuing in this way we build up our approximation to $f;$ the algorithm terminates when the residuals are sufficiently small.
%

We close this section by developing some useful results connected to the quasi-interpolation scheme.
First we recall that that the Fourier transform of $\psi$ is $\widehat{\psi} (t)=\exp(-2\pi^2 t^2).$
Next we develop the quasi-interpolants to the family of exponentials $e_{m}(x).$ We shall assume $h=\frac{1}{n}$ where $n=1,2,\ldots,$ then by definition we have
\[
\begin{aligned}
 Q_{\frac{1}{n}}(e_{m})(x):&=
 \sum_{\ell \in \ZZ}e_{m}\left(\frac{\ell}{n}\right)\psi(nx-\ell)=\sum_{j =0}^{n-1}\sum_{\ell \in \ZZ}e_{m}\left(\frac{n\ell+j}{n}\right)\psi(nx-(n\ell+j))
 \\
 &=\sum_{j =0}^{n-1}e_{m}\left(\frac{j}{n}\right)\sum_{\ell \in \ZZ}\psi(n(x-\ell)-j).
 \end{aligned}
\]
Let $\sigma(x)$ denote the infinite sum appearing in the final line above. We note  that
 $\sigma(x)$ is $1-$periodic and so we can consider its Fourier expansion
\[
\sigma(x) = \sum_{\ell \in \ZZ}\psi(n(x-\ell)-j)=\sum_{k=-\infty}^{\infty}\widehat{\sigma}_{k}e_{k}(x)\,\,\,\,{\rm{where}}\,\,\,\,
\widehat{\sigma}_{k}=\int_{0}^{1}\sigma(x)e_{-k}(x)dx.
\]
Using the periodicity of $\sigma$ together with an appropriate shift and scale in the variable of integration one can show that the Fourier coefficients are given by:
\[
\widehat{\sigma}_{k} ={1 \over n} e_{-k}\left(\frac{j}{n}\right) \widehat{\psi}\left(\frac{k}{n}\right).
\]
Substituting this back into the expression for $Q_{\frac{1}{n}}(e_{m})(x))$ we see that for $m\in \ZZ,$
\EQ{gnexp}{
\begin{aligned}
Q_{\frac{1}{n}} e_m (x) & =  \sum_{j=0}^{n-1} e_{m}\left(\frac{j}{n}\right)\left({1 \over n}\sum_{k=-\infty}^{\infty} e_{-k}\left(\frac{j}{n}\right) \widehat{\psi}\left(\frac{k}{n}\right)e_{k}(x)\right)
\\
& =  \sum_{k=-\infty}^{\infty}\widehat{\psi}\left(\frac{k}{n}\right) e_k(x) \left ( {1 \over n} \sum_{j=0}^{n-1}e_{j}\left(\frac{m-k}{n}\right) \right)
   \\
& =  \sum_{k=-\infty}^{\infty} \widehat{\psi}\left(\frac{nk+m}{n}\right) e_{m+nk} (x)= \sum_{k=-\infty}^{\infty}
\widehat{\psi}\left(k+\frac{m}{n}\right) e_{m+nk} (x).
\end{aligned}}
Using (\ref{gnexp}) it is straight forward to derive similar expressions for the family of trigonometric functions. In this paper we will focus only upon even functions and so, in preparation, we shall derive the equivalent expressions for the cosine family $c_{m}(x)=\cos(2\pi m x),$  $m=0,1,2\ldots.$ The same techniques can be used  for the sine family too, but this is not our main focus.
\begin{lemma} \label{quasitrig} For $n = 1,2,3,\ldots $ we can set $h=\frac{1}{n},$ and we have
\[
\begin{aligned}
&(i)\quad Q_{\frac{1}{n}}c_m  =  \sum_{k=-\infty}^\infty \widehat{\psi}\left(k+\frac{m}{n}\right)c_{m+nk}\quad m=0,1,\ldots\\
&(ii)\quad Q_{\frac{1}{n}}c_m=Q_{\frac{1}{n}}c_{m+jn}, \quad j \in \ZZ.
\end{aligned}\]
\end{lemma}
\begin{proof}
 For the first equation, we can use the identity $c_{m}=\frac{1}{2}(e_{m}+e_{-m})$ together with the fact that $\widehat{\psi}(-t)=\widehat{\psi}(t)$ to deduce
\begin{eqnarray*}
Q_{\frac{1}{n}}c_m  & = & {1 \over 2}\left(Q_{\frac{1}{n}}e_m+ Q_{\frac{1}{n}}e_{-m}\right)\\
& = &{1 \over 2} \left(\sum_{k=-\infty}^{\infty}
\widehat{\psi}\left(k+\frac{m}{n}\right) e_{m+nk}+\sum_{k=-\infty}^{\infty}
\widehat{\psi}\left(k-\frac{m}{n}\right) e_{-m+nk}\right)\\
& = &{1 \over 2} \left(\sum_{k=-\infty}^{\infty}
\widehat{\psi}\left(k+\frac{m}{n}\right) e_{m+nk}+\sum_{k=-\infty}^{\infty}
\widehat{\psi}\left(-k-\frac{m}{n}\right) e_{-m-nk}\right)\\
& = &\sum_{k=-\infty}^{\infty}
\widehat{\psi}\left(k+\frac{m}{n}\right)\frac{e_{m+nk}+e_{-m-nk}}{2}=\sum_{k=-\infty}^{\infty}
\widehat{\psi}\left(k+\frac{m}{n}\right)c_{m+nk}.
\end{eqnarray*}
The second equation is an immediate consequence of the first. \end{proof}

Our aim is to investigate the convergence rate of our proposed multi-level quasi-interpolation method. To set up the basic framework we shall assume, to begin with, that the target function $f\in \mathcal{N}$ and so possesses a Fourier series $\sum_{k=-\infty}^{\infty}\widehat{f}_{k}e_{k}.$ Let $h=\frac{1}{n}$ denote the spacing of the points at the first level. Then,
using (\ref{gnexp}), the quasi-interpolant is given by
\EQ{gngen}{
Q_{\frac{1}{n}} f  = \sum_{k \in \ZZ} \widehat{f}_{k}Q_{\frac{1}{n}}e_k (x)= \sum_{k=-\infty}^{\infty}
\widehat{f}_{k}\sum_{\ell \in \ZZ} \widehat{\psi}\left(\ell+\frac{k}{n}\right) e_{k+n\ell} (x).}

To describe the error at the first level we write  $E_{\frac{1}{n}}(f)=f-Q_{\frac{1}{n}}f.$
The error at the subsequent levels is defined recursively from here and we shall use the following notation. At level $p$ the multilevel error is given by
\[
M_{\frac{1}{n},p}(f)=E_{\frac{1}{2^{p-1}}\frac{1}{n}}M_{\frac{1}{n},p-1}(f),
\]
where we set $M_{h,0}=I$ to be the identity  operator so that $M_{\frac{1}{n},1}(f)=E_{\frac{1}{n}}(f).$

We begin our  investigation by measuring the norm of the quasi-interpolant (\ref{gngen}).
\EQ{initialbound}{
\|Q_{\frac{1}{n}} f\|\le \sum_{k \in \ZZ} \left|\widehat{f}_{k} \right|\sum_{\ell \in \ZZ} \widehat{\psi}\left(\ell+\frac{k}{n}\right).
}
Following Baxter \cite{baxter} an application of the Poisson summation formula yeilds
\[
 \sum_{\ell \in \ZZ} \widehat{\psi}\left(\ell+\frac{k}{n}\right)=\sum_{\ell \in \ZZ} \exp\left(-2\pi^{2}\left(\ell+\frac{k}{n}\right)^{2}\right)=\frac{1}{\sqrt{2\pi}}\sum_{\ell \in \ZZ}e^{-\frac{\ell^{2}}{2}}e^{ \frac{2\ell i\pi k}{n}},
\]
and we observe that this is a theta function of Jacobi type
\EQ{thetadef}{
\theta_{3}(z,q)=\sum_{\ell \in \ZZ}q^{\ell^{2}}e^{2 \ell iz}\,\,\,\,\,q \in \C \,\,\,{\rm{and}}\,\,\, |q|<1.
}
The following product function representation is found in \cite{gradshteyn} (8.181.2)
\EQ{prodreptheta}{
\theta_{3}(z,q)=\prod_{\ell=1}^{\infty}(1+2q^{2\ell-1}\cos(2z)+q^{2(2\ell-1)})(1-q^{2\ell})
}
If we choose $q = e^{-\frac{1}{2}}$ we can write
\[
\begin{aligned}
E(t):&=\sum_{\ell \in \ZZ} \widehat{\psi}\left(\ell+t\right)=
\frac{1}{\sqrt{2\pi}}\theta_{3}\left(\pi t,e^{-\frac{1}{2}}\right)\\
&=\frac{1}{\sqrt{2\pi}}\prod_{\ell=1}^{\infty}(1+2e^{-\ell+\frac{1}{2}}\cos(2\pi t)+e^{-(2\ell-1)})(1-e^{-\ell}).
\end{aligned}
\]

We observe that $E$ is $1-$periodic and, due to the product representation, it is decreasing on $[0,\frac{1}{2}]$ and increasing on $[\frac{1}{2},1]$ consequently $E$ attains its global max at zero. In view of these observation we can revisit (\ref{initialbound}) and deduce that $\|Q_{\frac{1}{n}} f\|\le \|f\|E(0).$ In view of (\ref{thetadef}) we have
\[
\begin{aligned}
E(0):=\sum_{\ell \in \ZZ}\exp\left(-2\pi^{2}\ell^{2}\right)=\theta_{3}(0,e^{-2\pi^{2}})=1+2e^{-2\pi^{2}}+2e^{-8\pi^{2}}+2e^{-18\pi^{2}}+\ldots
\end{aligned}
\]
where the right hand side are the leading terms in the expansion of (\ref{prodreptheta}) see \cite[16.38.5]{AS}. We can summarise the development above in the following theorem.
 \begin{proposition} \label{itbound}
Suppose $f \in \mathcal{N}$. Then, for $n = 1,2,\cdots$,
\[
\|Q_{\frac{1}{n}} f\|\le a\|f\|\,\,\,\, {\rm{where}} \,\, a = 1+3e^{-2\pi^{2}}=1+3\widehat{\psi}(1).
\]
 Consequently, setting $A = 1+a=2+3\widehat{\psi}(1)$ we have that
$$
\| E_{\frac{1}{n}}f \| = \| f-Q_{\frac{1}{n}} f \| \le A \| f \|\quad{\rm{and}}\quad \| M_{\frac{1}{n},p} f \| \le A^p \| f \|.
$$
\end{proposition}

\section{Convergence of the discrete algorithm} \label{fullalg}

In this section we will deal with even functions only, that is linear combinations of cosines. The proof for odd functions is the same, and the general case follows by decomposing a function into odd and even parts. To set the scene for what follows we clarify that the target function to which the algorithm is applied is $c_{m},$ the cosine function with a fixed frequency $m$ and the spacing between the data points at the first level is
given by $h=1/2^{\ell}$ for some $\ell\ge 2.$
We will split our investigation into two cases. First we deal with the situation when the initial $2^{\ell}$  sample points at level one is greater than the cosine frequency $m$ and secondly we examine the case when $2^{\ell}<m.$

\subsection{Cosine frequency  $<$ initial number of sample points:}

Assume that $m$ satisfies $hm<\frac{1}{2}$ or equivalently $m <2^{\ell-1}.$ In this setting
  we will prove a recursive formula for the multilevel approximation, aggregating factors of the size $\widehat{\psi}(2) \sim 10^{-35}$, in order to make the analysis tractable. With this in mind we fix the tolerance $\epsilon = 2 \widehat{\psi}(2)$. Final errors will have a contain a multiple of $\epsilon$.\\
$\diamond$ \textbf{Level one:} Using Lemma \ref{quasitrig},  the quasi-interpolant of $c_{m}$ in terms of $h$ is
\[
Q_{h}c_{m}=\sum_{k=1}^\infty\widehat{\psi}(hm-k)c_{m-\frac{k}{h}}+\widehat{\psi}(hm)c_{m} +\sum_{k=1}^\infty\widehat{\psi}(hm+k)c_{m+\frac{k}{h}}.
\]
Hence
\[
\begin{aligned}
E_{h}c_{m}&:=c_{m}-Q_{h}c_{m}\\
&=-\sum_{k=0}^\infty\widehat{\psi}(hm-k)c_{m-\frac{k}{h}}+\left(1-\widehat{\psi}(hm)\right)c_{m} -\sum_{k=0}^\infty\widehat{\psi}(hm+k)c_{m+\frac{k}{h}}.
\end{aligned}
\]
The plan of attack is to investigate the size of the residual error at each level of the algorithm by taking a central truncation of its series representation and examining this and the remainder separately.
 The amount of terms in the central truncation grows from level to level. For a typical level $p$  the truncation we have in mind consists of the contributions from $c_{m-\frac{2^{p}-j}{h}}$ for $j=0,\ldots 2^{p}-1$ and the contributions from $c_{m+\frac{j}{h}}$ also for $j=0,\ldots 2^{p}-1.$ At level one the split is as follows
\EQ{firsterror}{
E_{h}(c_{m}):= \underbrace{\overline{\alpha}_{0}^{(1)}c_{m-\frac{2}{h}}+\overline{\alpha}_{1}^{(1)}c_{m-\frac{1}{h}}
+\alpha_{0}^{(1)}c_{m}+\alpha_{1}^{(1)}c_{m+\frac{1}{h}}}_{=T_{h,1}c_{m}}+g_{1},
}
where $T_{h,1}c_{m}$ is the level one truncation whose coefficients are given by
\EQ{startcoefs}{
\begin{aligned}
&\overline{\alpha}_{0}^{(1)}=-\widehat{\psi}(hm-2),\,\,\quad \overline{\alpha}_{1}^{(1)}=-\widehat{\psi}(hm-1),\\
&\alpha_{0}^{(1)}=(1-\widehat{\psi}(hm)),\,\,\quad \alpha_{1}^{(1)}=-\widehat{\psi}(hm+1).
\end{aligned}
}
The function $g_{1}$ is the remainder term and is given by
\[
g_{1}=-\widehat{\psi}(hm+2)c_{m+\frac{1}{h}}-\sum_{k=3}^\infty\left(\widehat{\psi}(hm+k)c_{m+\frac{k}{h}}+
\widehat{\psi}(hm-k)c_{m-\frac{k}{h}}
\right).\]
In view of the fact that $hm<\frac{1}{2}$ we have the following bound for $g_{1}$:
\EQ{firstgbound}{
\|g_{1}\|\le \widehat{\psi}(2)
+\sum_{k=3}^{\infty}\widehat{\psi}(k)
+\widehat{\psi}\left(k-\frac{1}{2}\right)
\le 2\widehat{\psi}(2)=\epsilon.
}
We remark that for a crude bound on the truncation part we can evoke Proposition \ref{itbound}. In particular, this allows us to deduce that
\EQ{firstTbound}{
\|T_{h,1}c_{m}\|\le \|E_{h}c_{m}\|\le A.
}
$\diamond$ \textbf{Level two}. To consider the error at the second level, where the spacing is now $\frac{h}{2}$, we consider
\[
Q_{\frac{h}{2}}(c_{m}-Q_{h}c_{m})=Q_{\frac{h}{2}}T_{h,1}c_{m}+Q_{\frac{h}{2}}(g_{1}).
\]
Focussing on the truncation we can use  Lemma \ref{quasitrig} $(ii)$ to deduce that
\[
\begin{aligned}
Q_{\frac{h}{2}}T_{h,1}c_{m}&=\overline{\alpha}_{0}^{(1)}Q_{\frac{h}{2}}c_{m-\frac{2}{h}}
+\overline{\alpha}_{1}^{(1)}Q_{\frac{h}{2}}c_{m-\frac{1}{h}}
+\alpha_{0}^{(1)}Q_{\frac{h}{2}}c_{m}+\alpha_{1}^{(1)}Q_{\frac{h}{2}}c_{m+\frac{1}{h}}\\
&=(\overline{\alpha}_{0}^{(1)}+\alpha_{0}^{(1)})Q_{\frac{h}{2}}c_{m}+
(\overline{\alpha}_{1}^{(1)}+\alpha_{1}^{(1)})Q_{\frac{h}{2}}c_{m+\frac{1}{h}}.
\end{aligned}
\]
Now appealing to Lemma \ref{quasitrig} we have that
\[
\begin{aligned}
&Q_{\frac{h}{2}}c_{m}=
\sum_{k=1}^\infty\widehat{\psi}\left(\frac{hm}{2}-k\right)c_{m-\frac{2k}{h}}+\widehat{\psi}\left(\frac{hm}{2}\right)c_{m} +\sum_{k=1}^\infty\widehat{\psi}\left(\frac{hm}{2}+k\right)c_{m+\frac{2k}{h}}\\
&=\widehat{\psi}\left(\frac{hm}{2}-2\right)c_{m-\frac{4}{h}}+
\widehat{\psi}\left(\frac{hm}{2}-1\right)c_{m-\frac{2}{h}}+\widehat{\psi}\left(\frac{hm}{2}\right)c_{m}
+\widehat{\psi}\left(\frac{hm}{2}+1\right)c_{m+\frac{2}{h}}+g_{2}^{(0)},
\end{aligned}
\]
where
\[
g_{2}^{(0)}=\widehat{\psi}\left(\frac{hm}{2}+2\right)c_{m+\frac{4}{h}}+\sum_{k=3}^{\infty}
\left(\widehat{\psi}\left(\frac{hm}{2}-k\right)c_{m-\frac{2k}{h}}
+\widehat{\psi}\left(\frac{hm}{2}+k\right)c_{m+\frac{2k}{h}}\right).
\]
In a similar fashion we can show that
\[
\begin{aligned}
& Q_{\frac{h}{2}}c_{m+\frac{1}{h}}\\
&=\widehat{\psi}\left(\frac{hm-3}{2}\right)c_{m-\frac{3}{h}}+
\widehat{\psi}\left(\frac{hm-1}{2}\right)c_{m-\frac{1}{h}}
+\widehat{\psi}\left(\frac{hm+1}{2}\right)c_{m+\frac{1}{h}}+\widehat{\psi}\left(\frac{hm+3}{2}\right)c_{m+\frac{3}{h}}
+g_{2}^{(1)},
\end{aligned}
\]
where
\[
g_{2}^{(1)}=\widehat{\psi}\left(\frac{hm+5}{2}\right)c_{m+\frac{5}{h}}+\sum_{k=3}^{\infty}
\left(\widehat{\psi}\left(\frac{hm+1-2k}{2}\right)c_{m+
\frac{1-2k}{h}}
+\widehat{\psi}\left(\frac{hm+1+2k}{2}\right)c_{m+\frac{1+2k}{h}}\right).
\]
In view of the representations of $g_{2}^{(0)}$ and $g_{2}^{(1)}$ we can bound these functions in the same way as we did for $g_{1},$ see (\ref{firstgbound}), specifically $\| g_{2}^{(j)}\|\le  2\widehat{\psi}(2)=\epsilon,$   $(j=0,1).$
With this we can develop the level two error expression:
\[
\begin{aligned}
M_{h,2}c_{m}= E_{\frac{1}{2h}}(T_{h,1}c_{m}+g_{1})
= T_{h,1}c_{m}-Q_{\frac{h}{2}}T_{h,1}c_{m}+g_{1}-Q_{\frac{h}{2}}g_{1}=T_{h,2}c_{m}+g_{2}
\end{aligned},
\]
where the remainder is given by
\[
g_{2}=g_{1}-Q_{\frac{h}{2}}g_{1}-(\overline{\alpha}_{0}^{(1)}+\alpha_{0}^{(1)})g_{2}^{(0)}-
(\overline{\alpha}_{1}^{(1)}+\alpha_{1}^{(1)})g_{2}^{(1)},
\]
and the truncation
\[
\begin{aligned}
T_{h,2}c_{m}&=\overline{\alpha}_{0}^{(2)}c_{m-\frac{4}{h}}+\overline{\alpha}_{1}^{(2)}c_{m-\frac{3}{h}}+
\overline{\alpha}_{3}^{(2)}c_{m-\frac{2}{h}}+\overline{\alpha}_{4}^{(2)}c_{m-\frac{1}{h}}\\
&
+\alpha_{0}^{(2)}c_{m}+\alpha_{1}^{(2)}c_{m+\frac{1}{h}}+\alpha_{2}^{21)}c_{m+\frac{2}{h}}
+\alpha_{3}^{(2)}c_{m+\frac{3}{h}}.
\end{aligned}
\]
Following an inspection of the error expansion, the coefficients introduced above are
\[
\begin{aligned}
\overline{\alpha}_{j}^{(2)}&=-(\overline{\alpha}_{j}^{(1)}+\alpha_{j}^{(1)})
\widehat{\psi}\left(\frac{hm+j}{2}-2\right)\,\,\,\, (j=0,1),\\
\overline{\alpha}_{2^{1}+j}^{(2)}&=\overline{\alpha}_{j}^{(1)}-
(\overline{\alpha}_{j}^{(1)}+\alpha_{j}^{(1)})\widehat{\psi}\left(\frac{hm+j}{2}-1\right)\,\,\,\, (j=0,1),\\
\alpha_{j}^{(2)}&=\alpha_{j}^{(1)}-
(\overline{\alpha}_{j}^{(1)}+\alpha_{j}^{(1)})\widehat{\psi}\left(\frac{hm+j}{2}\right)\,\,\,\, (j=0,1),\\
\alpha_{2^{1}+j}^{(2)}&=-(\overline{\alpha}_{j}^{(1)}+\alpha_{j}^{(1)})\widehat{\psi}\left(\frac{hm+j}{2}+1\right)\,\,\,\, (j=0,1).\\
\end{aligned}
\]

To investigate the size of the remainder at level two we can proceed as follows.
\[
\begin{aligned}
\|g_{2}\|&\le \|g_{1}-Q_{\frac{h}{2}}g_{1}\|+\left(|\overline{\alpha}_{0}^{(1)}|+|\alpha_{0}^{(1)}|\right)\|g_{0}^{(1)}\|
+\left(|\overline{\alpha}_{1}^{(1)}|+|\alpha_{1}^{(1)}|\right)\|g_{1}^{(1)}\|\\
&\le \|E_{h}g_{1}\|+\left(|\overline{\alpha}_{0}^{(1)}|+|\alpha_{0}^{(1)}|
+|\overline{\alpha}_{1}^{(1)}|+|\alpha_{1}^{(1)}|\right)\epsilon \le A\|g_{1}\|+\|T_{h,1}c_{m}\|\epsilon\le 2A\epsilon.
\end{aligned}
\]


The preceding analysis of the first two levels provides sufficient insight to establish the following, more general result.
\begin{proposition} \label{propiter}
Let $m$ be the fixed frequency of the cosine $c_{m}$ and assume $m < 2^{\ell-1}$, for some $\ell \ge 2$. Let $h = 1/2^{\ell}$ denote the spacing at level one. Then
\beqns
M_{h,p} c_m & =  & T_{h,p} c_m + g_p,\quad {\rm{where}}\quad T_{h,p} c_m  =  \sum_{j=0}^{2^p-1} ( \overline{\alpha}_{j}^{(p)}  c_{m-\frac{2^{p}-j}{h}} +\alpha_{j}^{(p)} c_{m + \frac{j}{h}}),
\eeqns
where $\| g_p \| \le pA^{p-1} \epsilon$. The truncation coefficients are defined recursively.
Specifically, for $p=1$ the initial coefficients are given by (\ref{startcoefs}), and then for $p>1$ we have
 \EQ{coeffs}{
\begin{aligned}
\overline{\alpha}_{j}^{(p+1)}&=-(\overline{\alpha}_{j}^{(p)}+\alpha_{j}^{(p)})
\widehat{\psi}\left(\frac{hm+j}{2^{p}}-2\right)\,\,\,\, (j=0,1,\ldots,2^{p}-1),\\
\overline{\alpha}_{2^{p}+j}^{(p+1)}&=\overline{\alpha}_{j}^{(p)}-
(\overline{\alpha}_{j}^{(p)}+\alpha_{j}^{(p)})\widehat{\psi}\left(\frac{hm+j}{2^{p}}-1\right)\,\,\,\, (j=0,1,\ldots,2^{p}-1),\\
\alpha_{j}^{(p+1)}&=\alpha_{j}^{(p)}-
(\overline{\alpha}_{j}^{(p)}+\alpha_{j}^{(p)})\widehat{\psi}\left(\frac{hm+j}{2^{p}}\right)\,\,\,\, (j=0,1,\ldots,2^{p}-1),\\
\alpha_{2^{p}+j}^{(p+1)}&=-(\overline{\alpha}_{j}^{(p)}+\alpha_{j}^{(p)})\widehat{\psi}\left(\frac{hm+j}{2^{p}}+1\right)\,\,\,\, (j=0,1,\ldots,2^{p}-1).\\
\end{aligned}
}

\end{proposition}

\begin{proof} The result can be established by induction on $p.$ Indeed, the preceding error analysis for the level one case establishes the result for $p=1.$ Assuming the result for a general level $p$ (the inductive hypothesis) one can then  mimic the  methodology of the level two analysis to inductively establish the stated result for the $p+1$ level.
 \end{proof}


In order to make further progress we need to investigate how the size of truncation behaves as the algorithm proceeds through the levels. Thus we need to provide estimates on the sizes of the coefficients that appear in the truncation representation. In anticipation we provide the following elementary bounds on the $\widehat{\psi}$ multipliers. For $j=0,1,\ldots,2^{p}-1,$ and since $\widehat{\psi}(t)$ is decreasing for $t\ge0, $ we have that
\EQ{extreme1}{
\begin{aligned}
\widehat{\psi}\left(\frac{hm+j}{2^{p}}-2\right)&=\widehat{\psi}\left(2-\frac{hm+j}{2^{p}}\right)\le
  \widehat{\psi}\left(2-\frac{hm+2^{p}-1}{2^{p}}\right)\\
  &=\widehat{\psi}\left(1+\frac{1-hm}{2^{p}}\right)\le \widehat{\psi}\left(1\right)\quad {\rm{since}}\quad 0<hm<\frac{1}{2}.
  \end{aligned}
  }
  In addition we have the straight-forward bound
  \EQ{extreme2}{
  \widehat{\psi}\left(\frac{hm+j}{2^{p}}+1\right)\le \widehat{\psi}\left(1\right)\quad (j=0,1,\ldots, 2^{p}-1).
  }
  As a final comment we observe that since $0<\widehat{\psi}(t)\le 1$ for $t\ge 0$ then, for $j=0,1,\ldots 2^{p}-1$ we have the following crude bounds:
  \EQ{crude}{
  |\overline{\alpha}_{2^{p}+j}^{(p+1)}|\le |\overline{\alpha}_{j}^{(p)}|+|\alpha_{j}^{(p)}|\quad {\rm{and}}\quad
  |\alpha_{j}^{(p+1)}|\le |\overline{\alpha}_{j}^{(p)}|+|\alpha_{j}^{(p)}|.}

  In what follows we will assume that the algorithm is at or beyond the third level. At this stage  the cosine expansion of the truncation  $T_{h,p}c_{m}$  contains contributions from   cosines whose frequencies run from $m-\frac{2^{p}}{h}$ through to $m+\frac{2^{p}-1}{h}.$ We will bound the corresponding expansion coefficients at level $p$  in terms of the coefficients from previous levels and the results are set out in the following four lemmata. We start with the coefficients that correspond to the highest frequency cosines at level $p$.
%
%
%

\begin{lemma} \label{highcoeffbound}
Let $m < 2^{\ell-1}$, for some $\ell \ge 2$ and set $h = \frac{1}{2^{\ell}}.$ Then,  for $j=0,1,\ldots,2^{p-1}-1$ we have
$$
|\overline{\alpha}_{j}^{(p)}| \le \left(\overline{\alpha}_{j}^{(p-1)}+\alpha_{j}^{(p-1)}\right)\widehat{\psi}(1)\quad {\rm{and}}\quad
 |\alpha_{2^{p-1}+j}^{(p)}| \le \left(\overline{\alpha}_{j}^{(p-1)}+\alpha_{j}^{(p-1)}\right)\widehat{\psi}(1).
$$
\end{lemma}
\begin{proof} The above inequalities come directly from the expressions from the previous proposition  together with the bounds (\ref{extreme1}) and (\ref{extreme2}).
\end{proof}

The next two lemmata  deal with coefficients of the truncation that are associated with  mid-range frequencies of the cosines.

\begin{lemma} \label{midcoeffbounda}
Let $m < 2^{\ell-1}$, for some $\ell \ge 2$. Then, for $p \ge 3$ and for $j=0,1,\ldots,2^{p-2}-1$ we have
\[
\begin{aligned}
| \overline{\alpha}_{2^{p-1}+j}^{(p)}  |  \le   \mu_{a} \left(|\alpha_{j}^{(p-2)}|+|\overline{\alpha}_{j}^{(p-2)}|\right)\,\,\,\,{\rm{and}}\,\,\,\, | \alpha_{2^{p-2}+j}^{(p)} | & \le  \mu_{a} \left(|\alpha_{j}^{(p-2)}|+|\bar \alpha_{j}^{(p-2)}|\right)
\end{aligned}
\]
where $\mu_{a}=\left(\widehat{\psi}\left(\frac{1}{2}\right)+\widehat{\psi}(1) \right).$
\end{lemma}

\begin{proof} Using (\ref{coeffs}) we have
\[
\alpha_{2^{p-2}+j}^{(p)}  =   \alpha_{2^{p-2}+j}^{(p-1)}
\left( 1 - \widehat{\psi}\left(\frac{1}{2}+ \frac{hm+j}{2^{p-1}}\right) \right) -
\overline{\alpha}_{2^{p-2}+j}^{(p-1)} \widehat{\psi}\left(\frac{1}{2}+ \frac{hm+j}{2^{p-1}}\right)
\]
and thus
\EQ{midstart}{
\begin{aligned}
|\alpha_{2^{p-2}+j}^{(p)}|  &\le  | \alpha_{2^{p-2}+j}^{(p-1)}| +
|\overline{\alpha}_{2^{p-2}+j}^{(p-1)}| \widehat{\psi}\left (\frac{1}{2}\right )\\
&\le \left(|\overline{\alpha}_{j}^{(p-2)}|+|\alpha_{j}^{(p-2)}|\right)\widehat{\psi}(1)+|\overline{\alpha}_{2^{p-2}+j}^{(p-1)}| \widehat{\psi}\left (\frac{1}{2}\right ),
\end{aligned}
}
where the first bound follows since $\widehat{\psi}$ is decreasing and the second bound follows an application of Lemma \ref{highcoeffbound} with $p$ replaced by $p-1.$ Now appealing to (\ref{crude})  we have the following crude bound $
|\overline{\alpha}_{2^{p-2}+j}^{(p-1)}|\le
|\overline{\alpha}_{j}^{(p-2)}|+ |\alpha_{j}^{(p-2)}|.$
When substituted into (\ref{midstart}) this provides the desired result.
The proof of the bound for $|\overline{\alpha}_{2^{p-1}+j}^{(p)}|$ follows in the same fashion.
\end{proof}

\begin{lemma} \label{midcoeffboundb}
Let $m < 2^{\ell-1}$, for some $\ell \ge 2$. Then, for $p \ge 3$ and for $j=0,1,\ldots,2^{p-3}-1$ we have
\[
\begin{aligned}
| \overline{\alpha}_{2^{p-1}+2^{p-2}+j}^{(p)}  |  \le \mu_{b} \left(|\alpha_{j}^{(p-3)}|+|\overline{\alpha}_{j}^{(p-3)}|\right) \,\,\,\, {\rm{and}}\,\,\,\,
| \alpha_{2^{p-3}+j}^{(p)}  |  \le \mu_{b} \left(|\alpha_{j}^{(p-3)}|+|\overline{\alpha}_{j}^{(p-3)}|\right),
\end{aligned}
\]
where $\mu_{b}=\mu_{a}+\widehat{\psi}\left(\frac{1}{4}\right)\left(\widehat{\psi}(\sqrt{2})+\widehat{\psi}(1)\right).$
%
%

\end{lemma}
\begin{proof}

We begin with the bound for $\alpha_{2^{p-3}+j}^{(p)}$ where, from (\ref{coeffs}) we have
\[
\begin{aligned}
\alpha_{2^{p-3}+j}^{(p)}&=\alpha_{2^{p-3}+j}^{(p-1)}-\left(\alpha_{2^{p-3}+j}^{(p-1)}
+\overline{\alpha}_{2^{p-3}+j}^{(p-1)}\right)\widehat{\psi}\left(\frac{hm+j+2^{p-3}}{2^{p-1}}\right)\\
&=\left(1-\widehat{\psi}\left(\frac{1}{4}+\frac{hm+j}{2^{p-1}}\right)\right)\alpha_{2^{p-3}+j}^{(p-1)}-
\overline{\alpha}_{2^{p-3}+j}^{(p-1)}\widehat{\psi}\left(\frac{1}{4}+\frac{hm+j}{2^{p-1}}\right).
\end{aligned}
\]
Using the fact that $\widehat{\psi}$ is decreasing we have  the following bound
$
|\alpha_{2^{p-3}+j}^{(p)}| \le |\alpha_{2^{p-3}+j}^{(p-1)}|+\widehat{\psi}\left(\frac{1}{4}\right)|\overline{\alpha}_{2^{p-3}+j}^{(p-1)}|.$
Using the relevant result from  Lemma \ref{midcoeffbounda}  with $p-1$ instead of $p$ we can deduce that
$
|\alpha_{2^{p-3}+j}^{(p-1)}| \le \mu_{a}\left(|\alpha_{j}^{(p-3)}|+|\overline{\alpha}_{j}^{(p-3)}|\right).$
Now for $j=0,1,\dots,2^{p-3}$ we have that
$$
|\overline{\alpha}_{2^{p-3}+j}^{(p-1)}|\le \widehat{\psi}(1)\left(|\alpha_{2^{p-3}+j}^{(p-2)}|+|\overline{\alpha}_{2^{p-3}+j}^{(p-2)}|\right).
$$

We can now appeal to Lemma \ref{highcoeffbound} (with $p-2$ instead of $p$) to give $
|\alpha_{2^{p-3}+j}^{(p-2)}|\le \widehat{\psi}(1)\left(|\alpha_{j}^{(p-3)}|+|\overline{\alpha}_{j}^{(p-3)}|\right).
$
For the second term we can consider the following crude bound
$
|\overline{\alpha}_{2^{p-3}+j}^{(p-2)}|\le |\alpha_{j}^{(p-3)}|+|\overline{\alpha}_{j}^{(p-3)}|.
$
Bringing together the previous chain of development we have
\[
\begin{aligned}
|\alpha_{2^{p-3}+j}^{(p)}| &\le \mu_{a}\left(|\alpha_{j}^{(p-3)}|+|\overline{\alpha}_{j}^{(p-3)}|\right)
+\widehat{\psi}\left(\frac{1}{4}\right)\widehat{\psi}(1)\left(\widehat{\psi}(1)+1\right)\left(|\alpha_{j}^{(p-3)}|+|\overline{\alpha}_{j}^{(p-3)}|\right)\\
&\le \left(\mu_{a}+\widehat{\psi}\left(\frac{1}{4}\right)\left(\widehat{\psi}(\sqrt{2})+\widehat{\psi}(1)\right)\right)\left( |\alpha_{j}^{(p-3)}|+|\overline{\alpha}_{j}^{(p-3)}|\right)
\end{aligned}
\]
 The other inequality follows in a similar fashion. \end{proof}

We now conclude the bounding process by considering the low-frequency expansion coefficients.
\EQ{last}{
  \overline{\alpha}_{2^{p-1}+2^{p-2}+2^{p-3}+j}^{(p)}\quad {\rm{and}}\quad \alpha_{j}^{(p)}\quad
  (j=0,1,\ldots,2^{p-3}-1).}

\begin{lemma} \label{lowcoeffbound}
Let $m < 2^{\ell-1}$, for some $\ell \ge 2$. Then, for $p \ge 3$ and for $j=0,1,\ldots,2^{p-3}-1$ we have
\beqns
| \overline{\alpha}_{2^{p-1}+2^{p-2}+2^{p-3}+j}^{(p)} |  \le   \mu_{c}\left(|\alpha_{j}^{(p-3)}|+|\overline{\alpha}_{j}^{(p-3)}|\right), \,\,\,\, {\rm{and}}\,\,\,\,
| \alpha_{j}^{(p)}|  \le  \mu_{c}\left(|\alpha_{j}^{(p-3)}|+|\overline{\alpha}_{j}^{(p-3)}|\right),
\eeqns
where $\mu_{c}=\left(1-\widehat{\psi}\left(\frac{1}{4}\right)\right)
\left(1+\widehat{\psi}(1)-\widehat{\psi}\left(\frac{1}{2}\right)\right)+\widehat{\psi}\left(\frac{3}{2}\right)
\left(1+\widehat{\psi}(1)\right)$.
\end{lemma}
\begin{proof} We consider the bound on $| \alpha_{j}^{(p)}|.$ Using (\ref{coeffs}) we have
\[
| \alpha_{j}^{(p)}|= \left(1-\widehat{\psi}\left(\frac{hm+j}{2^{p-1}}\right)\right)|\alpha_{j}^{(p-1)}|+
\widehat{\psi}\left(\frac{hm+j}{2^{p-1}}\right)|\overline{\alpha}_{j}^{(p-1)}|.
\]
Now for $j=0,1,\ldots, 2^{p-3}-1,$ we have that $1-\widehat{\psi}\left(\frac{hm+j}{2^{p-1}}\right)\le 1-\widehat{\psi}\left(\frac{1}{4}\right)$ since $ 0<hm<\frac{1}{2}$. Therefore
\[
| \alpha_{j}^{(p)}|= \left(1-\widehat{\psi}\left(\frac{1}{4}\right)\right)|\alpha_{j}^{(p-1)}|+|
\overline{\alpha}_{j}^{(p-1)}|.
\]
Using (\ref{coeffs}) we also have
\[
|\alpha_{j}^{(p-1)}|=(|\alpha_{j}^{(p-2)}|+|\overline{\alpha}_{j}^{(p-2)}|)
\widehat{\psi}\left(\frac{hm+2}{2^{p-2}}-2\right),
\]
and, for  $j=0,1,\ldots, 2^{p-3}-1,$ we have that $\widehat{\psi}\left(\frac{hm+j}{2^{p-2}}-2\right) \le \widehat{\psi}\left(\frac{3}{2}+\frac{1-hm}{2^{p-2}}\right)\le \widehat{\psi}\left(\frac{3}{2}\right)$ since  $0<hm<\frac{1}{2}$.
Thus, we have
\[
|\alpha_{j}^{(p-1)}|\le \widehat{\psi}\left(\frac{3}{2}\right)\left(|\alpha_{j}^{(p-2)}|+|\overline{\alpha}_{j}^{(p-2)}|\right).
\]
Also, using the same argument as above one can easily show that
\[
| \alpha_{j}^{(p-1)}|= \left(1-\widehat{\psi}\left(\frac{1}{2}\right)\right)|\alpha_{j}^{(p-2)}|+|
\overline{\alpha}_{j}^{(p-2)}|.
\]

This allows us to deduce that
\EQ{almost}{
\begin{aligned}
| \alpha_{j}^{(p)}|&\le
\left[ \widehat{\psi}\left(\frac{3}{2}\right)+
\left(1-\widehat{\psi}\left(\frac{1}{4}\right)\right)\left(1-\widehat{\psi}\left(\frac{1}{2}\right)\right)
\right]|\alpha_{j}^{(p-2)}|\\
&+
\left(1-\widehat{\psi}\left(\frac{1}{4}\right)+\widehat{\psi}\left(\frac{3}{2}\right)\right)|\overline{\alpha}_{j}^{(p-2)}|.
\end{aligned}
}
To finish the bound we apply the following crude estimate
\[
|\alpha_{j}^{(p-2)}|\le |\alpha_{j}^{(p-3)}|+|\overline{\alpha}_{j}^{(p-3)}|,
\]
and the following from Lemma \ref{highcoeffbound}
\[
|\overline{\alpha}_{j}^{(p-2)}|\le \widehat{\psi}(1)\left(|\alpha_{j}^{(p-3)}|+|\overline{\alpha}_{j}^{(p-3)}|\right).
\]
Employing the above estimates in (\ref{almost}) delivers the required inequality.
 The second inequality follows in the same fashion. \end{proof}

The four lemmata above taken together provide bounds on the full set of expansion coefficients that appear in the level $p$ truncation of the multilevel error. The partitioned manner in which we developed these bounds now enables us to provide estimates for the norm of the truncation itself.
\[
\begin{aligned}
 \| T_{h,p} c_m \| &= \sum_{j=0}^{2^{p}-1}|\alpha_{j}^{(p)}|+|\overline{\alpha}_{j}^{(p)}|\\
 &\le \widehat{\psi}(1)\sum_{j=0}^{2^{p-1}-1}|\alpha_{j}^{(p-1)}|+|\overline{\alpha}_{j}^{(p-1)}|\quad (=\widehat{\psi}(1)\| T_{h,p-1} c_m \|\,\,\,{\rm{see}}\,\,{\rm{Lemma}}\,\,(\ref{highcoeffbound}))\\
 &+\mu_{a}\sum_{j=0}^{2^{p-2}-1}|\alpha_{j}^{(p-2)}|+|\overline{\alpha}_{j}^{(p-2)}|\quad (=\mu_{a}\| T_{h,p-2} c_m \|\,\,\,{\rm{see}}\,\,{\rm{Lemma}}\,\,(\ref{midcoeffbounda}))\\
&+\mu_{b}\sum_{j=0}^{2^{p-3}-1}|\alpha_{j}^{(p-3)}|+|\overline{\alpha}_{j}^{(p-3)}|\quad (=\mu_{b}\| T_{h,p-3} c_m \|\,\,\,{\rm{see}}\,\,{\rm{Lemma}}\,\,(\ref{midcoeffboundb}))\\
&+\mu_{c}\sum_{j=0}^{2^{p-3}-1}|\alpha_{j}^{(p-3)}|+|\overline{\alpha}_{j}^{(p-3)}|\quad (=\mu_{c}\| T_{h,p-3} c_m \|\,\,\,{\rm{see}}\,\,{\rm{Lemma}}\,\,(\ref{lowcoeffbound})).
\end{aligned}
\]

This development culminates in the main theorem of this part of the investigation.

\begin{theorem} \label{trate}
Let $m$ be the fixed frequency of the cosine $c_{m}$ and assume $m < 2^{\ell-1}$, for some $\ell \ge 2$. Let $h = 1/2^{\ell}$ denote the spacing at level one. Assume that $p \ge 3.$ Then
\[
 \| T_{h,p} c_m \| \le \widehat{\psi}(1)\| T_{h,p-1} c_m \| + a \| T_{h,p-2} c_m \|
 +b \| T_{h,p-3} c_m \|,
 \]
 where $\widehat{\psi}(1)< 3\times 10^{-9},$ $a = \mu_{a}<0.0072$ and $b = \mu_{b}+\mu_{c}<0.711.$ Consequently, for $p\ge 3$, there exists a constant $B>0$ independent of $m$ and $\ell$ such that
 \[
 \| T_{h,p} c_m \| \le B(0.9)^{p}.
 \]

 \end{theorem}

\begin{proof} The stated recursive bound is just a recasting of the estimate established above. To demonstrate the convergence of the truncation
 suppose that $M_p$ satisfies the relationship
$$
M_p = \widehat{\psi}(1) M_{p-1} + 0.0072M_{p-2}+0.711 M_{p-3},
$$
with $M_0=M_1=M_2=1$, and $M_{j} \le B(0.9)^j$, for $j=0,1,\cdots,p-1$. Then
\[
\begin{aligned}
M_p & \le  \widehat{\psi}(1) B(0.9)^{p-1} + 0.0072 B(0.9)^{p-2}+0.711B(0.9)^{p-3} \\
& =  B(0.9)^{p-3}\left(\widehat{\psi}(1)(0.9)^{2} +  0.0072(0.9)+0.711\right) \\
& \le  B(0.9)^p.
\end{aligned}
\]
If $\| T_{h,j} c_m \| \le M_j$, $j<p$, then $\| T_{h,p} c_m \|  \le M_p$, since
\beqns
M_p & = & \widehat{\psi}(1) M_{p-1} + 0.0072M_{p-2}+0.711 M_{p-3} \\
& \ge & \widehat{\psi}(1)\| T_{h,p-1} c_m \| + 0.0072 \| T_{h,p-2} c_m \|
 +0.711 \| T_{h,p-3} c_m \| \\
& \ge & \widehat{\psi}(1)\| T_{h,p-1} c_m \| + a \| T_{h,p-2} c_m \|
 +b \| T_{h,p-3} c_m \| \\
& \ge & \| T_{h,p} c_m \|,
\eeqns
from the established bound. Now, since (a straightforward calculation shows that) $\| T_{h,j} c_m \| \le 1$, $j=0,1,2$, $\| T_{h,p} c_k \| \le M_p \le B(0.9)^p$, $p \ge 3$. \end{proof}

\subsection{Cosine frequency $>$ initial number sample points}
Assume  that $m= 2^{\ell}+n$ where $0\le n < 2^{\ell},$ i.e., the cosine frequency $m$ is smaller than the $2^{\ell}$ points at level one but at level two the $2^{\ell+1}$  sample points will surpass $m.$ To shed some light on this situation we investigate the behaviour of the multilevel method if we start by sampling at the integers where $h=1.$ 

\begin{proposition} \label{highfreq}
Let $m=2^{\ell}+n$ for some $0 \le n < 2^{\ell}$. Then, for $1 \le p \le \ell+1$,
\beqns
M_{1,p} c_m & = & c_m - \sum_{j=0}^{p-1} M_{\frac{1}{2^{j+1}},p-1-j} Q_{\frac{1}{2^{j}}} c_{n (\md 2^j)}.
\eeqns
\end{proposition}
\begin{proof} Since $c_m(k) = c_0(k)$, $k \in \ZZ$, $Q_1 c_m(z) = Q_1c_0(z)$. Hence
\beqns
M_{1,1} c_m = E_1 c_m & = & c_m - Q_1 c_m  =  c_m - Q_1 c_0,
\eeqns
and the case $n=1$ is established since $\ell=0(\md 1)$ and $M_{\alpha,0}$ is the identity operator.
Now assume, for $1 \le p \le \ell$, that
\[
M_{1,p} c_m  =  c_m - \sum_{j=0}^{p-1} M_{\frac{1}{2^{j+1}},p-1-j}
Q_{\frac{1}{2^{j}}} c_{n (\md 2^j)}.
\]
Using Lemma~\ref{quasitrig}~(ii) we have $Q_{\frac{1}{2^{p}}} c_m = Q_{\frac{1}{2^{p}}} c_{n (\md 2^p)}$. Hence
\[
Q_{\frac{1}{2^{p}}} M_{1,p} c_m  =  Q_{\frac{1}{2^{p}}}c_{n (\md 2^p)} - \sum_{j=0}^{p-1} Q_{\frac{1}{2^{p}}} M_{\frac{1}{2^{j+1}},p-1-j} Q_{\frac{1}{2^{j}}} c_{n (\md 2^j)},
\]
so that $M_{1,p+1} c_m  =  M_{1,p} c_m - Q_{\frac{1}{2^{p}}} M_{1,p} c_m =$
\[
\begin{aligned}
& =  c_m - M_{\frac{1}{2^{p}},0} Q_{\frac{1}{2^{p}}} c_{n (\md 2^p)} \\
&- \sum_{j=0}^{p-1}\underbrace{\left[
  M_{\frac{1}{2^{j+1}},p-1-j} Q_{\frac{1}{2^{j}}} c_{n (\md 2^j)}-
  Q_{\frac{1}{2^{p}}} M_{\frac{1}{2^{j+1}},p-1-j} Q_{\frac{1}{2^{j}}}
  c_{n (\md 2^j)}
  \right]}_{= M_{\frac{1}{2^{j+1}},p-j} Q_{\frac{1}{2^{j}}} c_{n (\md 2^j)}} \\
& =  c_m - \sum_{j=0}^{p} M_{\frac{1}{2^{j+1}},p-j} Q_{\frac{1}{2^{j}}} c_{n (\md 2^j)},
\end{aligned}
\]
and the induction is completed. \end{proof}

\begin{corollary} \label{allits}
Let $m=2^{\ell}+n$ for some $0 \le n < 2^{\ell}$. Set $h = \frac{1}{2^{\ell}},$ then, for $p \ge \ell+2$,
\[
\begin{aligned}
M_{1,p} c_m  &=  M_{\frac{h}{4},p-(\ell+2)} c_m
-M_{\frac{h}{4},p-(\ell+2)}Q_{\frac{h}{2}} c_m- \sum_{j=0}^{\ell} M_{\frac{1}{2^{j+1}},p-(j+1)}Q_{\frac{1}{2^{j}}} c_{n (\md 2^j)}.
  \end{aligned}
\]
\end{corollary}
\begin{proof} From the previous proposition we have
\beqns
M_{1,\ell+1} c_m & = & c_m - \sum_{j=0}^{\ell} M_{\frac{1}{2^{j+1}},\ell-j} Q_{\frac{1}{2^{j}}} c_{n (\md 2^j)}.
\eeqns
Thus
\[
\begin{aligned}
M_{1,\ell+2} c_m & =  \underbrace{c_m - Q_{\frac{1}{2^{\ell+1}}} c_m}_{=c_{m}-Q_{\frac{h}{2}}c_{m}}
-\sum_{j=0}^{\ell} M_{\frac{1}{2^{j+1}},\ell+1-j} Q_{\frac{1}{2^{j}}} c_{n (\md 2^j)} \\
& =  M_{\frac{h}{4},0} c_m - M_{\frac{h}{4},0}Q_{\frac{h}{2}}c_{m}
-\sum_{j=0}^{\ell} M_{\frac{1}{2^{j+1}},\ell+2-j-1} Q_{\frac{1}{2^{j}}} c_{n (\md 2^j)},
\end{aligned}
\]
giving the required result for $p=\ell+2$, since $M_{\alpha,0}$ is the identity operator, for any positive $\alpha$. The result for $p \ge \ell+3$ follows directly from the definition of $M_{\frac{1}{2^{j}} ,p}$. \end{proof}

We now return to developing further the expressions for the multi-level error terms. In what follows we pursue a slightly different approach to the one that enabled the proof of Proposition \ref{propiter}.  This time we work with a different representation of the error expansion so that we partition the higher frequency contribution (entering the expansion at the most current level) from the remaining lower level frequencies.

\begin{proposition} \label{mqbnd}
Suppose $m < 2^\ell$, for some $\ell \in \NN$. Set $h = \frac{1}{2^{\ell}}$ then  for $p \ge 3$,
\[
\begin{aligned}
&M_{\frac{h}{2},p}  Q_{h} c_m \\
 & =  \alpha_0^1 M_{\frac{h}{2},p} c_m - \alpha_1^1 M_{\frac{h}{4},p-1}c_m + \alpha_2^1 M_{\frac{h}{8},p-2} c_m - \alpha_3^1 M_{\frac{h}{8},p-3} c_m\\
 & + \alpha_0^2 M_{\frac{h}{2},p} c_{m-\frac{1}{h}} - \alpha_1^2 M_{\frac{h}{4},p-1} c_{m-\frac{1}{h}}
 + \alpha_2^2 M_{\frac{h}{8},p-2} c_{m-\frac{1}{h}} -  \alpha_3^2 M_{\frac{h}{8},p-3} c_{m-\frac{1}{h}} \\
& +\beta_1^1 M_{\frac{h}{4},p-1} c_{m+\frac{1}{h}} - \beta_2^1 M_{\frac{h}{8},p-2} c_{m+\frac{1}{h}}
 + \beta_3^1 M_{\frac{h}{8},p-3} c_{m+\frac{1}{h}} \\
& +  \beta_1^2 M_{\frac{h}{4},p-1} c_{m-\frac{2}{h}} - \beta_2^2 M_{\frac{h}{8},p-2} c_{m-\frac{2}{h}}
 + \beta_3^2 M_{\frac{h}{8},p-3} c_{m-\frac{2}{h}}\\
& + \gamma_2^1 M_{\frac{h}{8},p-2} c_{m+\frac{2}{h}} - \gamma_3^1 M_{\frac{h}{8},p-3} c_{m+\frac{2}{h}} +
\gamma_2^2 M_{\frac{h}{8},p-2} c_{m-\frac{3}{h}} - \gamma_3^2 M_{\frac{h}{8},p-3} c_{m-\frac{3}{h}} \\
& + \gamma_2^3 M_{\frac{h}{8},p-2} c_{m+\frac{3}{h}} - \gamma_3^3 M_{\frac{h}{8},p-3} c_{m+\frac{3}{h}} +
\gamma_2^4 M_{2^{-(l+3)},p-2} c_{m-\frac{4}{h}} - \gamma_3^4  M_{\frac{h}{8},p-3} c_{m-\frac{4}{h}} \\
& - \delta_3^1 M_{\frac{h}{8},p-3} c_{m+\frac{4}{h}} - \delta_3^2 M_{\frac{h}{8},p-3} c_{m-\frac{5}{h}}
 - \delta_3^3 M_{\frac{h}{8},p-3} c_{m+\frac{5}{h}} - \delta_3^4 M_{\frac{h}{8},p-3} c_{m-\frac{6}{h}} \\
& - \delta_3^5 M_{\frac{h}{8},p-3} c_{m+\frac{6}{h}} - \delta_3^6 M_{\frac{h}{8},p-3} c_{m-\frac{7}{h}} - \delta_3^7 M_{\frac{h}{8},p-3} c_{m+\frac{7}{h}} - \delta_3^8 M_{\frac{h}{8},p-3} c_{m-\frac{8}{h}} + g_p,
\end{aligned}
\]
where $|\alpha_p^i|,|\beta_p^i|\le \hpsi(1)^p$, $i=1,2$, $|\gamma_p^i|\le \hpsi(1)^p$, $i=1,\cdots,4$, $|\delta_3^i|\le \hpsi(1)^3$, $i=1,\cdots,8$, and $\| g_p\| \le 2A^p \epsilon$.
\end{proposition}

\begin{proof} We have
\[
\begin{aligned}
Q_{h} c_m  &= \left[ \widehat{\psi}(hm ) c_{m} + \widehat{\psi}(hm-1 )c_{m-\frac{1}{h}}\right]+\left[ \widehat{\psi}(hm+1 ) c_{m+\frac{1}{h}} + \widehat{\psi}(hm-2 )c_{m-\frac{2}{h}}\right]+g_{0}\\
&=\left[ \alpha_{0}^{1} c_{m} + \alpha_{0}^{2}c_{m-\frac{1}{h}}\right]+\left[ \beta_{0}^{1}c_{m+\frac{1}{h}} + \beta_{0}^{2}c_{m-\frac{2}{h}}\right]+g_{0}
\end{aligned}
\]

where $|\alpha_0^1|,|\alpha_0^2| \le 1$, $|\beta_0^1|,|\beta_0^2| \le\hpsi(1)$, and as in the proof of Proposition~\ref{highfreq}, $\| g_0 \| \le \epsilon$.

Similarly we can show
\[
\begin{aligned}
Q_{\frac{h}{2}}Q_{h} c_m& =  \alpha_{0}^{1} Q_{\frac{h}{2}}c_{m}+\alpha_{0}^{2}Q_{\frac{h}{2}}c_{m-\frac{1}{h}}\\
&+\beta_{0}^{1}\Bigl[\left(\widehat{\psi}\left(\frac{hm-3}{2}\right)c_{m-\frac{3}{h}}
+\widehat{\psi}\left(\frac{hm+3}{2}\right)c_{m+\frac{3}{h}}\right)\\
&+\left(\widehat{\psi}\left(\frac{hm-1}{2}\right)c_{m-\frac{1}{h}}
+\widehat{\psi}\left(\frac{hm+1}{2}\right)c_{m+\frac{1}{h}}\right)+g_{1}^{1}\Bigr]\\
&+\beta_{0}^{2}\Bigl[\left(\widehat{\psi}\left(\frac{hm-4}{2}\right)c_{m-\frac{4}{h}}
+\widehat{\psi}\left(\frac{hm+2}{2}\right)c_{m+\frac{2}{h}}\right)\\
&+\left(\widehat{\psi}\left(\frac{hm-2}{2}\right)c_{m-\frac{2}{h}}
+\widehat{\psi}\left(\frac{hm}{2}\right)c_{m}\right)+g_{1}^{2}\Bigr]+Q_{\frac{h}{2}}g_{0},
\end{aligned}
\]
where $\|g_{1}^{1}\|,\|g_{1}^{2}\| \le \epsilon.$ These expressions allow us to deduce that
\[
\begin{aligned}
M_{\frac{h}{2},1}Q_{h}c_{m}&=Q_{h}c_{m}-Q_{\frac{h}{2}}Q_{h}= \alpha_{0}^{1}\underbrace{\left(c_{m}-Q_{\frac{h}{2}}c_{m}\right)}_{=M_{\frac{h}{2},1}c_{m}}
+\alpha_{0}^{2}\underbrace{\left(c_{m-\frac{1}{h}}
 -Q_{\frac{h}{2}}c_{m-\frac{1}{h}}\right)}_{=M_{\frac{h}{2},1}c_{m-\frac{1}{h}}}\\
 & -\alpha_{1}^{1}c_{m}-\alpha_{1}^{2}c_{m-\frac{1}{h}}+\beta_{1}^{1}c_{m+\frac{1}{h}}+\beta_{1}^{2}c_{m-\frac{2}{h}}\\
 &
-\gamma_{1}^{1}c_{m+\frac{2}{h}}-\gamma_{1}^{2}c_{m-\frac{3}{h}}-\gamma_{1}^{3}c_{m+\frac{3}{h}}
-\gamma_{1}^{4}c_{m-\frac{4}{h}}+\widetilde{g}_{1}+\underbrace{g_{0}-Q_{\frac{h}{2}}g_{0}}_{=M_{\frac{h}{2},1}g_{0}}\\
\end{aligned}
\]
where $\widetilde{g}_{1}=-(\beta_{0}^{1}g_{1}^{1}+\beta_{0}^{2}g_{1}^{1})$ with norm $\|\widetilde{g}_{1}\|\le 2\widehat{\psi}(1)\epsilon.$ Furthermore the coefficients appearing in the above are given by
\[
\begin{aligned}
 \alpha_{1}^{1}&=\beta_{0}^{2}\widehat{\psi}\left(\frac{hm}{2}\right),\quad 
 \alpha_{1}^{2}=\beta_{0}^{1}\widehat{\psi}\left(\frac{hm-1}{2}\right),\\\beta_{1}^{1}&=\beta_{0}^{1}\left(1-\widehat{\psi}\left(\frac{hm+1}{2}\right)\right),\quad 
 \beta_{1}^{2}=\beta_{0}^{2}\left(1-\widehat{\psi}\left(\frac{hm-2}{2}\right)\right),\\
 \gamma_{1}^{1}&=\beta_{0}^{2}\widehat{\psi}\left(\frac{hm+2}{2}\right),\,\,\,
 \gamma_{1}^{2}=\beta_{0}^{1}\widehat{\psi}\left(\frac{hm-3}{2}\right),\\
  \gamma_{1}^{3}&=\beta_{0}^{1}\widehat{\psi}\left(\frac{hm+3}{2}\right),\,\,\,
 \gamma_{1}^{4}=\beta_{0}^{2}\widehat{\psi}\left(\frac{hm-4}{2}\right),
 \end{aligned}
 \]
so that $|\alpha_{1}^{1}|,|\alpha_{1}^{2}|,|\beta_{1}^{1}|,|\beta_{1}^{1}|\le \widehat{\psi}(1)$ and $|\gamma_{1}^{1}|,\ldots,|\gamma_{1}^{4}|\le (\widehat{\psi}(1))^{2}.$ In summary we conclude that multi-level error at level $p=1$ is given by
\[
\begin{aligned}
M_{\frac{h}{2},1}Q_{h}c_{m}&=\alpha_{0}^{1}M_{\frac{h}{2},1}c_{m}+\alpha_{0}^{2}M_{\frac{h}{2},1}c_{m-\frac{1}{h}}
 -\alpha_{1}^{1}c_{m}-\alpha_{1}^{2}c_{m-\frac{1}{h}}+\beta_{1}^{1}c_{m+\frac{1}{h}}+\beta_{1}^{2}c_{m-\frac{2}{h}}\\
 &
-\gamma_{1}^{1}c_{m+\frac{2}{h}}-\gamma_{1}^{2}c_{m-\frac{3}{h}}-\gamma_{1}^{3}c_{m+\frac{3}{h}}
-\gamma_{1}^{4}c_{m-\frac{4}{h}}+g_{1},
\end{aligned}
\]
where $g_{1}=\widetilde{g}_{1}+M_{\frac{h}{2},1}g_{0}$ which has norm
$
\|g_{1}\|\le 2\widehat{\psi}(1)\epsilon+A\|g_{0}\|\le (2\widehat{\psi}(1)+A)\epsilon\le 2A\epsilon.
$

If we repeat the above chain of reasoning we get the following expression for the level two ($p=2$) error.

\[
\begin{aligned}
M_{\frac{h}{2},2}Q_{h}c_{m}=&=  \alpha_0^1 M_{\frac{h}{2},2} c_m -
\alpha_1^1 M_{\frac{h}{4},1}c_m + \alpha_2^1 c_m  \\
& + \alpha_0^2 M_{\frac{h}{2},2} c_{m-\frac{1}{h}}
 - \alpha_1^2 M_{\frac{h}{4},1} c_{m-\frac{1}{h}}
  + \alpha_2^2 c_m \\
& +\beta_1^1 M_{\frac{h}{4},1} c_{m+\frac{1}{h}}
- \beta_2^1 c_{m+\frac{1}{h}} +  \beta_1^2 M_{\frac{h}{4},1} c_{m-\frac{2}{h}}
 - \beta_2^2 c_{m-\frac{2}{h}} \\
& + \gamma_2^1 c_{m+\frac{2}{h}} + \gamma_2^2 c_{m-\frac{3}{h}} +
\gamma_2^3 c_{m+\frac{3}{h}} + \gamma_2^4 c_{m-\frac{4}{h}}  \\
& - \delta_2^1 c_{m+\frac{4}{h}} + \delta_2^2 c_{m-\frac{5}{h}} + \delta_2^3 c_{m+\frac{5}{h}} +
\delta_2^4 c_{m-\frac{6}{h}}\\
&- \delta_2^5 c_{m+\frac{6}{h}} + \delta_2^6 c_{m-\frac{7}{h}} + \delta_2^7 c_{m+\frac{7}{h}} + \delta_2^8 c_{m-\frac{8}{h}} + g_2,
\end{aligned}
\]
where $|\alpha_2^1|,|\alpha_2^2|,|\beta_2^1|,|\beta_2^2|,|\gamma_2^1|,\cdots,|\gamma_2^4| \le (\hpsi(1))^2$, $|\delta_2^1|,\cdots,|\delta_2^8| \le (\hpsi(1))^3$, and $\| g_2 \| \le  A \|  g_1 \|+8(\hpsi(1))^2 \epsilon \le A(A+\hpsi(1)) \epsilon + 8(\hpsi(1))^2 \epsilon = (A(A+\hpsi(1)) + 8(\hpsi(1))^2) \epsilon\le 2A^{2}\epsilon$.

A third iteration of this argument will introduce $16$ new high frequencies $c_{m+\frac{8}{h}},\ldots c_{m+\frac{15}{h}}$
and $c_{m-\frac{9}{h}},\ldots c_{m-\frac{16}{h}}$ but these cosines are multiplied by coefficients of order $(\hpsi(1))^4 = \hpsi(2)$. This leads to the level three $(p=3)$ error given by
\[
\begin{aligned}
M_{\frac{h}{2},3}  Q_{h} c_m & =  \alpha_0^1 M_{\frac{h}{2},3} c_m - \alpha_1^1 M_{\frac{h}{4},2}c_m +
 \alpha_2^1 M_{\frac{h}{8},1} c_m-\alpha_3^1 c_m\\
 & + \alpha_0^2 M_{\frac{h}{8},3} c_{m-\frac{1}{h}} - \alpha_1^2 M_{\frac{h}{2},2} c_{m-\frac{1}{h}} + M_{\frac{h}{8},1} \alpha_2^2c_{m-\frac{1}{h}} -\alpha_3^2 c_{m-\frac{1}{h}} \\
& +\beta_1^1 M_{\frac{h}{4},2} c_{m+\frac{1}{h}} - \beta_2^1 M_{\frac{h}{8},1} c_{m+\frac{1}{h}} + \beta_3^1 c_{m+\frac{1}{h}} \\
& +  \beta_1^2 M_{\frac{h}{4},2} c_{m-\frac{2}{h}} - \beta_2^2 M_{\frac{h}{8},1} c_{m-\frac{2}{h}} + \beta_3^2 c_{m-\frac{2}{h}}\\
& + \gamma_2^1 M_{\frac{h}{8},1} c_{m+\frac{2}{h}} - \gamma_3^1 c_{m+\frac{2}{h}} + \gamma_2^2
M_{\frac{h}{8},1} c_{m-\frac{3}{h}} - \gamma_3^2 c_{m-\frac{3}{h}} \\
& + \gamma_2^3 M_{\frac{h}{8},1} c_{m+\frac{3}{h}} - \gamma_3^3 c_{m+\frac{3}{h}} + \gamma_2^4 M_{\frac{h}{8},1}
 c_{m-\frac{4}{h}} - \gamma_3^4  c_{m-\frac{4}{h}} \\
& - \delta_3^1 c_{m+\frac{4}{h}} + \delta_3^2 c_{m-\frac{5}{h}} + \delta_3^3 c_{m+\frac{5}{h}} +
\delta_3^4 c_{m-\frac{6}{h}} \\
& - \delta_3^5 c_{m+\frac{6}{h}} + \delta_3^6 c_{m-\frac{7}{h}} + \delta_3^7 c_{m+\frac{7}{h}} + \delta_3^8 c_{m-\frac{8}{h}} \\
& + g_3,
\end{aligned}
\]
where $|\alpha_3^1|,|\alpha_3^2|,|\beta_3^1|,|\beta_3^2|,|\gamma_3^1|,\cdots,|\gamma_3^4| \le (\hpsi(1))^3$, $\| g_3\| \le 16 \hpsi(2) + A \| g_2 \| +32 \hpsi(2) (\hpsi(1))^3 \le 8 \epsilon + A(A(A+\hpsi(1)) + 8(\hpsi(1))^2) \epsilon + 16(\hpsi(1))^3 \epsilon = (A(A(A+\hpsi(1)) + 8(\hpsi(1))^2+16(\hpsi(1))^3+8) \epsilon \le 2 A^3 \epsilon$.

From this point on, we introduce no new terms in our sums, so that the process of iterating through the levels for $p \ge 3$ leads to the representation given in the proposition.
\end{proof}


\begin{corollary} \label{lowcoeffrate1}
Suppose $m < 2^{\ell}$, for some $\ell \in \NN$. Let $h = \frac{1}{2^{\ell}}$ then  for $p \ge 3$,
$$
\| M_{\frac{h}{2},p}  Q_{h} c_m \| \le 3B(0.9)^p+2A^p \epsilon.
$$
\end{corollary}
\begin{proof} From the previous proposition we have $|\alpha_i^j|,|\beta_i^j|,|\gamma_i^j|,|\delta_i^j| \le \hpsi(1)^i$, for all appropriate values of $i$ and $j$. We also have, from Proposition~\ref{itbound} and Theorem~\ref{trate}, $\|T_{h,p} c_m \| \le B(0.9)^p$, whenever $m<2^{\ell-1}$. Thus we can bound each term in the sum in the statement of Proposition~\ref{mqbnd} to see that
\beqns
\| M_{\frac{h}{2},p}  Q_{h} c_m \| & \le & 2B\left[
 (0.9)^p + 2 \hpsi(1) (0.9)^{p-1} + 4 \hpsi(1)^2 (0.9)^{p-2} + 8 \hpsi(1)^3 (0.9)^{p-3}\right] + \| g_p \| \\
& \le & 2B(0.9)^p \left(1+
\left(\frac{20\widehat{\psi}(1)}{9}\right)+\left(\frac{20\widehat{\psi}(1)}{9}\right)^2+\left(\frac{20\widehat{\psi}(1)}{9}\right)^3\right) + 2A^p \epsilon \\
& \le & 3B(0.9)^p + 2A^p \epsilon.
\eeqns
\end{proof}

We can combine this result with Proposition~\ref{highfreq} to get
\begin{corollary} \label{highfreqbnd}
Let $\ell \in \NN$ and $m=2^\ell+n$ for some $0 \le n < 2^{\ell}$. Set $h = \frac{1}{2^{\ell}}$ then, for $p \ge \ell+5$,
$$
\| M_{1,p} c_m \| \le 31B(0.9)^{p-\ell-2}+(p+2)A^{p} \epsilon.
$$
\end{corollary}
\begin{proof} From Corollary~\ref{allits} we have
\beqn
M_{1,p} c_m & =  & M_{\frac{h}{4},p-(\ell+2)} c_m - M_{\frac{h}{4},p-(\ell+2)} Q_{\frac{h}{2}} c_m
\nonumber \\
&& \hspace{1cm} - \sum_{j=0}^{\ell} M_{\frac{1}{2^{j+1}},p-j-1} Q_{\frac{1}{2^{j}}} c_{n (\md 2^j)}.
\eeqn
Since $p-(\ell+2)\ge 3,$ we can use Proposition~\ref{propiter} ($\ell+2$ in place of $\ell$) together with  Theorem~\ref{trate}, to yield
\beqn
\| M_{\frac{h}{4},p-(\ell+2)} c_m \| & \le & B(0.9)^{p-\ell-2}+(p-\ell-2)A^{p-\ell-2}\epsilon. \label{topterm}
\eeqn
As $2^{\ell} \le m < 2^{\ell+1}$, using Lemma~\ref{quasitrig}~(ii) we see that $Q_{\frac{h}{2}} c_m = Q_{\frac{h}{2}} c_{m-\frac{2}{h}}$. Because $2^{\ell+1}-m < 2^{\ell}$, from Corollary~\ref{lowcoeffrate1},
\beqn
\| M_{\frac{h}{4},p-(\ell+2)} Q_{\frac{h}{2}} c_m \| & = & \| M_{\frac{h}{4},p-(\ell+2)} Q_{\frac{h}{2}} c_{m-\frac{2}{h}} \| \nonumber \\
& \le & 3B(0.9)^{p-\ell-2} + 2A^{p-\ell-2} \epsilon. \label{nextterm}
\eeqn
Similarly, directly from Corollary~\ref{lowcoeffrate1}, for $0 \le j \le \ell$,
\beqn
\| M_{\frac{1}{2^{j+1}}} Q_{\frac{1}{2^{j}}} c_{n (\md 2^j)} \|
& \le & 3B(0.9)^{p-j-1} + 2A^{p-j-1} \epsilon. \label{restterm}
\eeqn
Substituting (\ref{topterm})--(\ref{restterm}) into (\ref{fullser})   we see that $\| M_{1,p} c_m \|  \le$
\beqns
 & \le  & B(0.9)^{p-\ell-2}+(p-\ell-2)A^{p-\ell-2}\epsilon + \sum_{j=0}^{\ell+1} (3B(0.9)^{p-j-1} + 2A^{p-j-1}\epsilon) \\
& \le & B(0.9)^{p-\ell-2}+(p-\ell-2)A^{p-\ell-2}\epsilon + 3B(0.9)^{p-\ell-2} \sum_{j=0}^{\ell+1} (0.9)^{j} +
2 A^{p-\ell-2}\epsilon \sum_{j=0}^{\ell+1}A^{j}\\
& \le & B(0.9)^{p-\ell-2} +(p-\ell-2)A^{p-\ell-2}\epsilon+ 30B(0.9)^{p-\ell-2} +
 2 A^{p-\ell-2}\left(\frac{A^{\ell+2}-1}{A-1}\right) \epsilon \\
& \le & 31 B(0.9)^{p-\ell-2} + 2(p-\ell)A^{p}\epsilon,
\eeqns
since $A-1 =1+3\widehat{\psi}(1) \ge 1$. \end{proof}

\medskip
We now prove the main theorem of the paper.

\begin{theorem}
Let $f= \sum_{k=0}^\infty \widehat{f}_{k} c_k \in W_s$. Then, if $s \ge 1$ and $1/2<t<s$,
\beqns
\| f - M_{1,p} \|_\infty \le \left ( 31B(1+ D(s))(0.9)^p + C(t) A^p 2^{-(p-2) (s-t)} +  {2 \over \sqrt{3}} p^{3/2} A^p \epsilon \right ) \| f \|_s,
\eeqns
where
$$
D(s)  =  \left ( \sum_{l=0}^{p-3} (0.9)^{-2\ell-4} 2^{-2\ell(s-1/2)} \right )^{1/2}, \quad
{\rm and} \quad C(t) = \left (\sum_{k=2^{p-2}}^\infty k^{-2t} \right )^{1/2}.
$$
\end{theorem}

\begin{proof} We first split the error into three components
\[
\begin{aligned}
&M_{1,p} (f)  =  \sum_{k=0}^\infty \widehat{f}_{k} M_{1,p} (c_k) \\
& =  \widehat{f}_{0} M_{1,p} (c_0) + \sum_{\ell=0}^{p-3} \sum_{m=0}^{2^\ell-1}\widehat{f}_{2^{\ell}+m} M_{1,p} (c_{2^{\ell}+m}) 
+ \sum_{k=2^{p-2}}^\infty  \widehat{f}_{k}M_{1,p} (c_k). 
\end{aligned}
\]
Using a similar argument to that of Proposition~\ref{highfreqbnd} we can show that
\beqn
\| M_{1,p} (c_0) \| & \le & (2B \widehat{\psi}(1) (0.9)^{p-2} +3 A^{p-1} \epsilon) \nonumber \\
& \le & 31B (0.9)^p + {1 \over \sqrt{3}} p^{3/2} A^p. \label{firstterm}
\eeqn

We bound the final term of the error expression above using Proposition~\ref{itbound} followed by an application of the Cauchy-Schwarz inequality. For any $s>t>1/2$,
\beqn
\lefteqn{\left \| \sum_{k=2^{p-2}}^\infty  \widehat{f}_{k}M_{1,p} (c_k) \right \|_\infty} \nonumber \\
& \le & \sum_{k=2^{p-2}}^\infty  \widehat{f}_{k} \|M_{1,p} (c_k) \|
\le A^p \sum_{k=2^{p-2}}^\infty  |\widehat{f}_{k}| \nonumber
\le A^p 2^{-(p-2) (s-t)} \sum_{k=2^{p-2}}^\infty  k^{s-t} |\widehat{f}_{k}| \nonumber \\
& \le & A^p 2^{-(p-2) (s-t)} \left (\sum_{k=2^{p-2}}^\infty  k^{2s} |\widehat{f}_{k})|^2 \right )^{1/2} \left (\sum_{k=2^{p-2}}^\infty k^{-2t} \right )^{1/2} \le C(t) A^p 2^{-(p-2) (s-t)} \| f \|_s. \nonumber \\
&& \label{spl3}
\eeqn
We now bound the middle term of the error expression.
\beqn
\left \| \sum_{\ell=0}^{p-3} \sum_{m=0}^{2^\ell-1}\widehat{f}_{2^{\ell}+m} M_{1,p} (c_{2^{\ell}+m}) \right \|_\infty
 & \le & \sum_{m=0}^{2^{\ell}-1} |\widehat{f}_{2^{\ell}+m}| \| M_{1,p} (c_{2^{\ell}+m}) \| \nonumber \\
& \le & \sum_{\ell=0}^{p-3} \sum_{m=0}^{2^{\ell}} |\widehat{f}_{2^{\ell}+m}| (31 B(0.9)^{p-\ell-2} + 2(p-\ell)A^{p}\epsilon) \nonumber \\
& = & \sum_{\ell=0}^{p-3} (31B(0.9)^{p-\ell-2}  + 2(p-\ell)A^{p}\epsilon) \sum_{m=0}^{2^{\ell}} |\widehat{f}_{2^{\ell}+m}|. \label{totsum}
\eeqn
Now, using the Cauchy-Schwarz inequality again we obtain
\beqns
\sum_{m=0}^{2^{\ell}-1} |\widehat{f}_{2^{\ell}+m}| & \le & 2^{-\ell s} \sum_{m=0}^{2^{\ell}} (2^{\ell}+m)^s |\widehat{f}_{2^{\ell}+m}| \le 2^{-\ell s} 2^{\ell/2} S_{\ell},
\eeqns
where $S_{\ell} = \left ( \sum_{m=0}^{2^{\ell}} (2^{\ell}+m)^{2s} |\widehat{f}_{2^{\ell}+m}|^2 \right )^{1/2}$, $\ell=0,1,\cdots,p-3$. Substituting into (\ref{totsum}) we have
\beqn
\left \| \sum_{\ell=0}^{p-3} \sum_{m=0}^{2^{\ell}-1} \widehat{f}_{2^{\ell}+m}(c_{2^{\ell}+m} - M_{1,p} c_{2^{\ell}+m}) \right \|_\infty & \le & \sum_{\ell=0}^{p-3} (31 B(0.9)^{p-\ell-2}  + 2(p-\ell)A^{p}\epsilon) 2^{-\ell(s-1/2)} S_{\ell} \nonumber \\
&& \hspace{-2cm} \le  31B(0.9)^p \sum_{\ell=0}^{p-3} (0.9)^{-\ell-2} 2^{-\ell(s-1/2)} S_{\ell} + 2\epsilon \sum_{\ell=0}^{p-3} (p-\ell)A^{p} S_{\ell}. \label{simform}
\eeqn
Using the Cauchy-Schwarz inequality a final time
\beqns
\sum_{\ell=0}^{p-3} 3(p-l) S_{\ell} & \le & 3 A^{p} \left ( \sum_{\ell=0}^{p-3} (p-\ell)^2 \right )^{1/2} \left ( \sum_{\ell=0}^{p-3}S_{\ell}^2 \right )^{1/2} \\
& \le & {1 \over \sqrt{3}} p^{3/2} A^p  \| f \|_s,
\eeqns
and
\beqns
\sum_{\ell=0}^{p-3} (0.9)^{-\ell-2} 2^{-\ell(s-1/2)} S_{\ell}
& \le & \left ( \sum_{\ell=0}^{p-3} (0.9)^{-2\ell-4} 2^{-2\ell(s-1/2)} \right )^{1/2} \left ( \sum_{\ell=0}^{p-3}S_{\ell}^2 \right )^{1/2} \\
& \le & D(s) \| f \|_s,
\eeqns
which is finite if $s \ge 1$. Putting the last two equations into (\ref{simform}), together with (\ref{spl3}) and (\ref{firstterm}) we obtain the stated result.
\end{proof}

 Since we are interested in approximation of smooth functions we can take $s$ to be as large as we please in the result above. In this case we obtain the following corollary.
 \begin{corollary}
 Let $f= \sum_{k=0}^\infty \widehat{f}_k c_k \in W_s$ for $s \ge 3$. Then,
\beqns
\| f - M_{1,p} \|_\infty \le \left ( E(p) (0.9)^p +  {1 \over \sqrt{3}} p^{3/2} A^p \epsilon \right ) \| f \|_s,
\eeqns
for some constant $E(s)>0$.
 \end{corollary}
\begin{proof} The result follow immediately from the previous theorem observing that for $t=1$, $A/2^{s-t} \le A/4 \le 0.9$ since $A<3$ and $s \ge 3$. \end{proof}

\section{Numerical experiments} \label{numerics}

In this section we look at three numerical examples. The first two are $f=c_1$ and $f=c_9$ so that we can observe the algorithm treating frequencies similarly once we have more points than the degree of the cosine. The third example will be of the smooth function $f(x)=\exp(c_1)$ to observe the scheme on a function with a full cosine expansion. We compare the convergence of the two single modes with coefficient
$$
m_p = (1-2 \hpsi(1))(1-2 \hpsi(2^{-1})) \cdots (1-\hpsi(2^{-p})),
$$
which arises naturally in the multilevel iteration applied to $c_1$; see coefficient $\alpha_0^{(1)}$ in (\ref{startcoefs}). Our hypothesis is that this sequence is more indicative of the convergence rate of the algorithm than the theoretical convergence rate of $(0.9)^p$.

\begin{table}[h]
\centering
\begin{tabular}{|c|c|c|c|c|} \hline
Level $p$  & $\| M_{1,p} c_1\|_\infty$ & $\| M_{1,p} c_9\|_\infty$ & $m_p$ & $\| M_{1,p} f \|_\infty$  \\ \hline \hline
1 & 9.9 (-1) & 2.0 & 9.9 (-1) & 1.2\\ \hline
2 & 7.0 (-1) & 1.0 & 7.0 (-1) & 1.1 (-1)  \\ \hline
3 & 1.9 (-1) & 1.3 & 1.9 (-1) & 4.4 (-1)\\ \hline
4 & 1.4 (-2) & 1.8 & 1.4 (-2) & 9.1 (-2) \\ \hline
5 & 2.4 (-4) & 1.0 & 2.6 (-3) & 8.5 (-3)\\ \hline
6 & 1.2 (-6) & 8.0 (-1) & 1.3 (-6) & 3.1 (-4) \\ \hline
7 & 2.8 (-9) & 2.6 (-1) & 1.5 (-9) & 3.4 (-6) \\ \hline
8 & 1.0 (-10) & 2.4 (-2) & 4.6 (-13) & 1.6 (-8) \\ \hline
9 & 2.9 (-12) & 5.7 (-4) & 3.4 (-17) & 7.3 (-11) \\ \hline
10 & 5.6 (-14)& 3.5 (-6) & 6.5 (-22) & 2.7 (-12) \\ \hline
\end{tabular}

\caption{Comparison of algorithm on $c_1$ and $c_9$ and the sequence $m_p$, and the multilevel approximation error for $f=\exp(c_1)$}
\label{twofreq}
\end{table}

We can see in Table~\ref{twofreq} that the decay rate for $f=c_1$ is almost identically that for $m_p$, up until $p=7$. At this stage the decay rate for $\| M_{1,p} c_1 \|_\infty$ is governed by the coefficient of other cosines than $c_1$ in the expansion for $M_{1,p} c_1$. We can see that the decay of $\| M_{1,p} c_9 \|_\infty$ lags that of $\| M_{1,p} c_1 \|_\infty$ by around 4 levels of iteration, but then the rates track consistently. This observation reflects the analysis of the previous sections, where convergence happens when the sampling rate is high enough.

In the same table we record the error for approximation of $f(x)=\exp(c_1)$. We see that the convergence is similar to that of $\| M_{1,n} c_1\|_\infty$, lagging by around 1 level, where due to the full Fourier expansion, there is a little less predictability in the early iterations.

\end{document}